\documentclass{amsart}
\usepackage[utf8]{inputenc}
\usepackage{amsmath,amssymb}
\usepackage{comment}
\usepackage{xcolor}
\usepackage{url}
\usepackage[colorlinks=true,allcolors=blue,backref=page]{hyperref}
\usepackage[noabbrev,capitalize,nameinlink]{cleveref}
\usepackage{enumerate}
\usepackage{todonotes}
\usepackage[nocompress]{cite}

\newtheorem{theorem}{Theorem}

\newtheorem{lemma}[theorem]{Lemma}

\newtheorem*{claim*}{Claim}

\theoremstyle{remark}
\newtheorem*{remark*}{Remark}

\theoremstyle{definition}
\newtheorem{conjecture}[theorem]{Conjecture}
\newtheorem{assumption}[theorem]{Assumption}

\numberwithin{theorem}{section}

\renewcommand{\phi}{\varphi}
\renewcommand{\emptyset}{\varnothing}

\renewcommand{\leq}{\le}
\renewcommand{\geq}{\ge}

\newcommand{\eps}{\varepsilon}

\renewcommand{\le}{\leqslant}
\renewcommand{\ge}{\geqslant}
\renewcommand{\P}{\mathbb{P}}

\newcommand{\R}{\mathbb R}
\newcommand{\C}{\mathbb C}
\newcommand{\E}{\mathbb E}

\newcommand{\sgn}{\operatorname{sgn}}

\title{On random polynomials with an intermediate number of real roots}

\author{Marcus Michelen}
\address{Department of Mathematics, Statistics and Computer Science \\ University of Illinois, Chicago  \\ 851 S.\ Morgan St. \\ Chicago, IL 60607 \\ USA}
\email{michelen@uic.edu}
\thanks{M. Michelen is supported in part by NSF grants DMS-2137623 and DMS-2246624.}
\author{Sean O'Rourke}
\address{Department of Mathematics\\ University of Colorado Boulder\\ Campus Box 395\\ Boulder, CO 80309-0395\\USA}
\email{sean.d.orourke@colorado.edu}
\thanks{S. O'Rourke is supported in part by NSF CAREER grant DMS-2143142.}


\begin{document}
	\begin{abstract}
        For each $\alpha \in (0, 1)$, we construct a bounded monotone deterministic sequence $(c_k)_{k \geq 0}$ of real numbers so that the number of real roots of the random polynomial $f_n(z) = \sum_{k=0}^n c_k \eps_k z^k$ is $n^{\alpha + o(1)}$ with probability tending to one as the degree $n$ tends to infinity, where $(\eps_k)$ is a sequence of i.i.d.\ (real) random variables of finite mean satisfying a mild anti-concentration assumption.  In particular, this includes the case when $(\eps_k)$ is a sequence of i.i.d.\ standard Gaussian or Rademacher random variables.  This result confirms a conjecture of O.\ Nguyen from 2019.  More generally, our main results also describe several statistical properties for the number of real roots of $f_n$, including the asymptotic behavior of the variance and a central limit theorem.  
	\end{abstract}	
\maketitle

\section{Introduction} \label{sec:intro}
Consider the random polynomial 
\begin{equation} \label{def:fn}
	f_n(z) = \sum_{k=0}^n c_k \eps_k z^k, 
\end{equation} 
where $(c_k)$ is a deterministic sequence of real numbers and $(\eps_k)$ is a sequence of independent and identically distributed (i.i.d.) real-valued random variables (e.g., $(\eps_k)$ is a sequence of i.i.d.\ standard Gaussian random variables).  Beginning with Bloch and P\'olya's work in the 1930s \cite{bloch1932roots}, researchers have studied the roots (both real and complex) of $f_n$ and used random polynomials for a variety of applications to many different areas of mathematics \cite{bloch1932roots,MR1690355,MR1290398,MR0856019,MR1574980}.  Among the earliest questions is how the number of real roots of $f_n$ depends on the sequence $(c_k)$. 

In the case when $(\eps_k)$ is a sequence of i.i.d.\ Gaussian random variables, the Kac--Rice formula and other tools which take advantage of special properties of the Gaussian distribution are available and allow one to write the moments of the number of real roots as a deterministic integral; see, for instance, \cite{MR2185759,MR2552864,MR0007812,MR1575245,MR1290398,MR1308023,MR0084888} and references therein.  Recent developments have focused on universality results, where $(\eps_k)$ is a sequence of i.i.d.\ non-Gaussian random variables, and one often wishes to understand how the behavior of the roots depends of the distribution of $(\eps_k)$.  We refer the reader to \cite{MR3602403,MR3439098,MR4367414,MR3846831,MR3262481,MR3771497,MR3280673,MR0139199,MR0286157,MR0234512,MR0234513} and references therein; however, this list is far from complete. 

In this note, we focus on the number of real roots of $f_n$.  
Some of the most widely studied models for this problem include the Kac model ($c_k \equiv 1$) and the Weyl model ($c_k = 1/\sqrt{k!}$).  
Let $R_n$ denote the number of real roots of $f_n$.  We review some previously studied models here:

\begin{itemize}
\item In the case when $(\eps_k)$ is a sequence of i.i.d.\ standard Gaussian random variables and $c_k = 1$ for all $k \geq 1$, it was shown by Kac \cite{MR0007812} that\footnote{Here, $o(1)$ denotes a term which tends to zero as $n$ tends to infinity; see Section \ref{sec:notation} for more details about the asymptotic notation used here and throughout the paper.}
\[ \E R_n = \frac{2}{\pi} \log n + C + o(1), \]
where $C$ is an explicit constant.  See \cite{nguyen2016number,erdos1956number,ibragimov1971average,ibragimov1971expected}  and the references therein for generalizations to other distributions. 

\item In the case when $(\eps_k)$ is a sequence of i.i.d.\ standard Gaussian random variables and $c_k = \frac{1}{\sqrt{k!}}$, it is known \cite{MR1290398,MR3439098} that 
\[ \E R_n = \left( \frac{2}{\pi} + o(1) \right) \sqrt{n}. \]
\item When $(\eps_k)$ is a sequence of i.i.d.\ Rademacher random variables and $c_1 = 1/2$, $c_k = 1/(k!)^k$ for $k \geq 2$, Littlewood and Offord \cite{MR1574980} showed that $f_n$ has all real roots with probability $1$.  
\item When $c_k = \exp(-k^\beta/2)$ and $(\eps_k)$ is a sequence of i.i.d.\ standard Gaussian random variables, Schehr and Majumdar \cite{schehr2009condensation} observed the following phase transition: \begin{equation*}
    \E R_n = \begin{cases}
        \left(\frac{2}{\pi} + o(1)\right)\log n & \text{ if }\beta \in [0,1) \\
        \left(\frac{2}{\pi} + o(1) \right)\sqrt{\frac{\beta - 1}{\beta}}n^{\beta/2} &\text{ if } \beta \in (1,2) \\
        (1 + o(1)) n &\text{ if }\beta > 2
    \end{cases}\,.
\end{equation*}
Notably, as $\beta$ increases, the expected number of real roots ``jumps'' from order $\log n$ to $n^{1/2+\eps}$ as $\beta$ crosses $1$.
\end{itemize}

Based on the above results, the following conjecture\footnote{The conjecture was formulated in a workshop at the American Institute of Mathematics and can be found online at \url{http://aimpl.org/randpolyzero/}.}
is natural.  
\begin{conjecture}[O. Nguyen] \label{conjecture}
Let $(\eps_k)$ be an i.i.d.\ sequence of standard Gaussian random variables.  Given $\alpha \in (0, 1)$, there exists a bounded sequence $(c_k)$ of real numbers so that $\E R_n = n^{\alpha + o(1)}$.  
\end{conjecture}
Here, it is important that $c_k$ only depends on $k$; if $c_k$ is allowed to also depend on $n$, it becomes straightforward to establish Conjecture \ref{conjecture}.  

In this note, we confirm Conjecture \ref{conjecture} with a monotone decreasing sequence.

\begin{theorem} \label{thm:conjecture}
Let $(\eps_k)$ be an i.i.d.\ sequence of standard Gaussian random variables.  For any $\alpha \in (0, 1)$, there exists a monotone bounded sequence $(c_k)$ of real numbers so that $\E R_n = n^{\alpha + o(1)}$.  
\end{theorem}

More generally, our main result below (Theorem \ref{thm:main}) establishes several other statistical properties of $R_n$ for our construction, including a central limit theorem, and shows these properties hold for many other coefficient distributions, beyond the Gaussian case.  

We now turn to the construction of the deterministic sequence $(c_k)$.  
Fix $\alpha \in (0,1)$.  Define the sequence $m_j := 2\lfloor j^{1/\alpha} \rfloor$ and the interval of integers $I_j := (m_{j-1},m_{j}]$; we set $m_{-1} = -1$ so that way the sequence $(I_j)_{j \geq 0}$ forms a partition of $\{0,1, 2, \ldots\}$.  We now define the sequence $(c_k)_{k \geq 0}$ via  
\begin{equation} \label{def:cmj}
	c_{k} = \exp(- 2^j) \qquad\text{ if }\qquad k \in I_j\,.
\end{equation} 
Note that the sequence $(c_k)$ depends on $\alpha$, but we do not denote this dependence in our notation.  

We will assume $(\eps_k)$ is a sequence of i.i.d.\ random variables satisfying the following properties.
\begin{assumption} \label{assumption}
 Let $(\eps_k)_{k \geq 0}$ be a sequence of real-valued i.i.d.\ random variables which satisfy the following assumptions: 
\begin{enumerate}[(i)]
\item \label{assump:symmetric} $\P(\eps_k > 0) =: p \in (0,1)$

\item \label{assump:subgaussian} $\mathbb{E} |\eps_k| \leq 1$ 

\item \label{assump:anticoncentration} There exists a constant $ C_0 > 0$ so that $\P( |\eps_k| \leq t) \leq C_0 t$ for all $t \geq 0$. 
\end{enumerate}
\end{assumption}

Since the number of real roots is invariant under scaling the polynomial, Condition \eqref{assump:subgaussian} of Assumption \ref{assumption} simply asserts that the random variables lie in $L^1$. 
Condition \eqref{assump:anticoncentration} implies that $\eps_k$ is non-zero with probability $1$.  In particular, both the standard Gaussian distribution and the Rademacher distribution satisfy Assumption \ref{assumption}.  We note that Condition \eqref{assump:subgaussian} and \eqref{assump:anticoncentration} can be relaxed considerably, although we do not pursue this direction.  

Fix $\alpha \in (0, 1)$.  Define the random polynomial $f_n$ of degree $n$ by \eqref{def:fn}, where $(c_k)$ is the sequence defined in \eqref{def:cmj} and $(\eps_j)$ is a sequence of i.i.d.\ real-valued random variables satisfying Assumption \ref{assumption}.  
Note that $f_n$ has $(1 + o(1)) (n/2)^\alpha$ ``blocks'' of coefficients, i.e.,\ if we set $j_\ast = \max\{j : I_j  \cap \{0,1,\ldots,n\} \neq \emptyset \}$ then $j_\ast = (1 + o(1)) (n/2)^\alpha$.  The idea will be that the terms indexed by the sequence $(m_j)_{j \geq 0}$ will dictate the behavior of the polynomial with high probability.  As such, we will let $S_n$ denote the number of sign changes among the sequence $(\eps_{m_j})_{j = 0}^{j_\ast-1} \cup (\eps_n)$ 
(note that by Condition \eqref{assump:anticoncentration} $\eps_j$ is non-zero with probability $1$).  

We say an event $\mathcal{E}$ (which depends on $n$) holds with \emph{overwhelming probability} if for every $\kappa > 0$ there exists a constant $C_\kappa > 0$ so that $\P(\mathcal{E}) \geq 1 - C_\kappa n^{-\kappa}$ for all $n > C_\kappa$.  

\begin{theorem}\label{thm:main}
Fix $\alpha \in (0, 1)$.  Define the random polynomial $f_n$ by \eqref{def:fn}, where $(c_k)$ is the sequence defined in \eqref{def:cmj} and $(\eps_k)$ is a sequence of i.i.d.\ random variables satisfying Assumption \ref{assumption}.  
Let $R_n$ be the number of real roots of $f_n$, and let $S_n$ be the number of sign changes among the sequence $(\eps_{m_j})_{j = 0}^{j_\ast - 1} \cup (\eps_n)$.  Then there exists an absolute constant $C > 0$ so that the event $|R_n - 2S_n| \leq C (\log n)^{1/\alpha}$ holds with overwhelming probability.  
In particular, there are constants $c_p$ and $c_p'$ depending only on $p$ from Condition \eqref{assump:symmetric} so that $\E R_n \sim c_p (n/2)^\alpha$, $\operatorname{Var} R_n \sim c_p'(n/2)^\alpha$, $R_n / (c_p(n/2))^\alpha \to 1$ almost surely as $n \to \infty$, and $$\frac{R_n - \E R_n}{\sqrt{\operatorname{Var} R_n}} \longrightarrow \mathcal{N}(0,1)$$
in distribution as $n \to \infty$, where $\mathcal{N}(0,1)$ denotes the standard normal distribution.  
\end{theorem}

Theorem \ref{thm:conjecture} follows immediately from Theorem \ref{thm:main}.  We outline the proof in the following short subsection and prove Theorem \ref{thm:main} in Section \ref{sec:proof}.  

\subsection{Outline and sketch of the proof of Theorem \ref{thm:main} }

The intuition for why $f_n$ has $n^{\alpha + o(1)}$ many real roots is that for each $t \in \R$ we have one of two types of behavior: either the $m_j$ term dominates, in which case we do not have a real root with high probability; or we are transitioning from when the $m_j$ term dominates to when the $m_{j+1}$ term dominates, in which case we have a real root if and only if there is a sign change among $\{\eps_{m_j}, \eps_{m_{j+1}}\}$ with high probability.  Since we have chosen $m_j$ to always be an even integer, the behavior on the negative real axis is quite similar to that on the positive real axis, and so we only focus on the positive real axis.

In order to confirm this heuristic, we switch to exponential coordinates and define $g_n(t) = f_n(e^t)$ and first write a deterministic lemma that drives all of our comparisons between terms (Lemma \ref{lemma:increment-calculation}). Here we introduce disjoint intervals $[a_j,b_j]$ where the $m_j$ term of $g_n$ dominates on $[a_j,b_j]$.  We then convert our deterministic Lemma \ref{lemma:increment-calculation} to a probabilistic statement showing that in fact the $m_j$ term of $g_n$ does dominate on $[a_j,b_j]$ with high probability (Lemma \ref{lemma:anticoncentration}).  This immediately implies there are no real roots in $[a_j,b_j]$ with high probability (Lemma \ref{lem:no-root}) and that if $\{\eps_{m_j},\eps_{m_{j+1}}\}$ are the same sign then there is no real root in $[b_j,a_{j+1}]$ with high probability (Lemma \ref{lem:no-change}).  To handle the case when we \emph{do} have a sign change among $\{\eps_{m_j},\eps_{m_{j+1}}\}$, Lemma \ref{lemma:anticoncentration} shows that $g_n$ has a sign change on $[b_j,a_{j+1}]$ with high probability thus showing that there is a real root; in order to show that there is only one real root, we show that the derivative of a rescaled version of $g_n$ is nonzero on $[b_j,a_{j+1}]$ with high probability (Lemma \ref{lem:yes-change}). Finally, a simple application of Rouch\'e's Theorem shows that there are not too many real roots near the origin (Lemma \ref{lem:rouche}), allowing us to use the asymptotic results we have proven up until that point.

\subsection{Notation} \label{sec:notation}
Starting in Section \ref{sec:proof}, unless otherwise noted, we use asymptotic notation under the assumption that $j \to \infty$; this differs from Section \ref{sec:intro}, where we used asymptotic notation under the assumption that $n \to \infty$.  For two sequences $(a_j)$ and $(b_j)$, we write $a_j = O(b_j)$ if there exists a constant $C > 0$ so that $|a_j| \leq C b_j$ for all $j > C$.  If the constant $C$ depends on other parameters, e.g., $C = C_\alpha$, we denote this with subscripts, e.g., $a_j = O_\alpha(b_j)$.  We write $a_j = o(b_j)$ to denote that $ \lim_{j \to \infty} a_j/b_j = 0$.  The notation $a_j \sim b_j$ denotes that $a_j = (1 + o(1)) b_j$.  We use $i$ as an index (i.e., $i$ does not denote the imaginary unit).  We let $C, c$ be positive constants (i.e., they do not depend on $n$ or $j$) that may change from one occurrence to the next.

\section{Proof of Theorem \ref{thm:main}} \label{sec:proof}

It will often be more convenient to work in exponential coordinates, and  we write $g_n(t) := f_n(e^t).$   Fix $\alpha \in (0,1)$ and set $\beta := \min\{1/2,\frac{1-\alpha}{2\alpha}\}$.  For each $j \geq 1$, define $$a_j :=  \alpha j^{1-1/\alpha} 2^{j-2}(1 + j^{-\beta})\qquad \text{ and } \qquad b_j := \alpha j^{1 - 1/\alpha} 2^{j-1}(1 - j^{-\beta})\,.$$
It follows that that there exists a constant $C > 0$ (depending only on $\alpha$) so that $a_j < b_j$ for all $j > C$ and $(a_j)_{j > C}$ and $(b_j)_{j > C}$ are increasing sequences.  
Further define the functions $\varphi_j(t):= c_{m_j} e^{t m_j}$ and set $j_\ast = \max\{j : I_j  \cap \{0,1,\ldots,n\} \neq \emptyset \}$ so we may write 
$$g_n(t) = \sum_{ j < j_\ast } \varphi_j(t) \sum_{i = 0}^{m_j - m_{j-1} - 1} \eps_{m_j - i} e^{-it} + \sum_{i = m_{j_\ast - 1} + 1}^{n} c_{i}\eps_{i} e^{it}\,.$$

We write $g_n$ in this form to note that the blocks for $j < j_\ast$ are ``complete'' while the last block may potentially be ``incomplete.''
We prove Theorem \ref{thm:main} via a sequence of lemmas, starting with a deterministic lemma showing that $\varphi_j$ dominates on the interval $[a_j,b_j]$: 

\begin{lemma}\label{lemma:increment-calculation}
    There are constants $C, c > 0$ depending only on $\alpha$ so that for all $j \geq 1$ we have 
    \begin{equation} \label{eq:left-increment}
    \log\left(\frac{\varphi_{j-1}(t)}{\varphi_j(t)} \right) \leq -c 2^j j^{-\beta} + C \end{equation}
    for any $t \geq a_j$ and  
    \begin{equation} \label{eq:right-increment}
    \log\left(\frac{\varphi_{j+1}(t)}{\varphi_j(t)} \right) \leq -c 2^j j^{-\beta} + C  \end{equation}
    for all $t \leq b_j$. 
    In particular, there exist constants $C', c' > 0$ depending only on $\alpha$ so that
    \begin{equation} \label{eq:left-right-increment-large}
    	\log\left(\frac{\varphi_{j-1}(t)}{\varphi_j(t)} \right) \leq -c' 2^j j^{-\beta} \quad \text{ and } \quad \log\left(\frac{\varphi_{j+1}(t)}{\varphi_j(t)} \right) \leq -c' 2^j j^{-\beta} 
    \end{equation}
    for all $j \geq C'$, where the first bound holds for all $t \geq a_j$ and the second for all $t \leq b_j$.  
\end{lemma}
\begin{proof}
    For $t \geq a_j$, we compute \begin{align*}
        \log\left(\frac{\varphi_{j-1}(t)}{\varphi_j(t)} \right) &\leq -2^{j-1} + 2^j + a_j\left(2\lfloor (j-1)^{1/\alpha} \rfloor  -  2\lfloor j^{1/\alpha} \rfloor\right) \\
        &= 2^{j-1} - \left(\alpha j^{1 - 1/\alpha} 2^{j-2}(1 + j^{-\beta})\right)\left(\frac{2j^{1/\alpha - 1}}{\alpha} + O_\alpha(1 + j^{1/\alpha - 2})\right) \\
        &= -2^{j-1}\left(j^{-\beta} + O_\alpha(j^{1 - 1/\alpha} + j^{-1}) \right)\,.
    \end{align*}
    By the choice of $\beta$ we note that  $\beta < 1$ and $\beta < 1/\alpha - 1$ and so we see that there are constants $C, c > 0$ so that  \eqref{eq:left-increment} holds.  
    Similarly, for $t \leq b_j$, we compute  
    \begin{align*}
        \log\left(\frac{\varphi_{j+1}(t)}{\varphi_j(t)} \right) &\leq 2^j  -2^{j+1} + b_j\left( 2\lfloor (j+1)^{1/\alpha} \rfloor - 2\lfloor j^{1/\alpha} \rfloor \right) \\
        &= -2^j + \left(\alpha j^{1 - 1/\alpha} 2^{j-1}(1 - j^{-\beta})\right)\left(\frac{2j^{1/\alpha - 1}}{\alpha} + O_\alpha(1 + j^{1/\alpha - 2})\right) \\
        &= -2^{j}\left(j^{-\beta} + O_\alpha(j^{1 - 1/\alpha} + j^{-1}) \right)\,.
    \end{align*}
    The choice of $\beta$ again establishes \eqref{eq:right-increment}.  The bounds in \eqref{eq:left-right-increment-large} follow from \eqref{eq:left-increment} and \eqref{eq:right-increment}. 
\end{proof}

We now extend this deterministic bound to show that the $m_j$ term of our random polynomial dominates on $[a_j,b_j]$:

\begin{lemma} \label{lemma:anticoncentration}
There exist constants $C, c > 0$ (depending only on $\alpha$ and $C_0$) so that the following holds.  
\begin{itemize}
\item For any $j \geq 1$, with probability at least $1 - C \exp(-c j^2)$,

\begin{equation} \label{eq:anticoncentration-full}
	\min_{a_j \leq s \leq b_j} \left( \frac{c_{m_j}e^{m_js} |\eps_{m_j}|}{2} - \sum_{i \neq m_j} |\eps_{i}|  c_i e^{is} \right) > 0. 
\end{equation}
\item For any $j \geq 1$, with probability at least $1 - C \exp(-c j^2)$, 
\begin{equation} \label{eq:anticoncentration-left}
	\inf_{a_j \leq s} \left(  \frac{c_{m_j}e^{m_js} |\eps_{m_j}|}{2} - \sum_{i < m_j} |\eps_{i}|  c_i e^{is} \right) > 0. 
\end{equation}
\item For any $j \geq 1$, with probability at least $1 - C \exp(-c j^2)$, 
\begin{equation} \label{eq:anticoncentration-right}
	\inf_{s \leq b_j} \left(  \frac{c_{m_j}e^{m_js} |\eps_{m_j}|}{2}- \sum_{i > m_j} |\eps_{i}|  c_i e^{is} \right) > 0. 
\end{equation} 
\item With probability at least $1 - C \exp(-c j_\ast^2)$, 
\begin{equation} \label{eq:highest-term}
    \inf_{a_{j_\ast} \leq s} \left(c_n e^{ns} |\eps_{n}| - \sum_{i < n} |\eps_i| c_i e^{is} \right) > 0\,.
\end{equation}

\end{itemize}
\end{lemma}
\begin{proof}
We prove \eqref{eq:anticoncentration-full}; the proofs of \eqref{eq:anticoncentration-left}, \eqref{eq:anticoncentration-right} and \eqref{eq:highest-term} are similar, and we omit the details. 
We first show that the $m_j$ term dominates all others in its block.  To this end, since $c_i = c_{m_j}$ for all $m_{j-1}+1 \leq i \leq m_j$, we write $$c_{m_j} e^{m_j s}|\eps_{m_j}|/4 - \sum_{i = m_{j-1}+1}^{m_j - 1} |\eps_i| c_i e^{i s} \stackrel{\text{d}}{=} c_{m_j}e^{m_j s}\left( \frac{|\eps_{m_j}|}{4} - \sum_{i = 1}^{m_j - m_{j-1} -1} |\eps_i| e^{-is} \right),  $$
where $\stackrel{\text{d}}{=}$ denotes equality in distribution.  
Using Condition \eqref{assump:subgaussian}, we bound $$ \E \max_{a_j \leq s \leq b_j}\sum_{i = 1}^{m_j - m_{j-1} -1} |\eps_i| e^{-is} \leq 2 e^{- a_j} = e^{-\Omega_\alpha(j^2)}. $$
By Markov's inequality along with Condition \eqref{assump:anticoncentration}, we can bound \begin{align}
    \P\left( \min_{a_j \leq s \leq b_j} \left( c_{m_j} e^{m_j s}\frac{|\eps_{m_j}|}{4} - \sum_{i \in I_j \setminus j} |\eps_i| c_i e^{i s} \right)  \leq 0 \right)
    &\leq \P(|\eps_{m_j}| \leq C e^{-c j^2 / 2}) + e^{-\Omega_\alpha(j^2)} \nonumber \\
    &= e^{-\Omega_\alpha(j^2)}\,. \label{eq:dominate-block}
\end{align}
To handle the coefficients not in $I_j$, first note that \begin{align*}
    \sum_{i \notin I_j} |\eps_i| \frac{c_{i} e^{i s}}{c_{m_j}e^{m_js} } \leq \sum_{r \neq j}\frac{\varphi_r(s)}{\varphi_j(s)} \sum_{i \in I_r} |\eps_i|\,.
\end{align*}
By iteratively applying Lemma \ref{lemma:increment-calculation} and using Condition \eqref{assump:subgaussian}, we have \begin{align*}
    \E \max_{a_j \leq s \leq b_j} \left(\sum_{r \neq j}\frac{\varphi_r(s)}{\varphi_j(s)} \sum_{i \in I_r} |\eps_i|\right) 
    & \leq C m_j e^{-c j^2} + \sum_{r > j} C|I_r| e^{-c r^2} \\
    &= e^{-\Omega_\alpha(j^2)}\,.
\end{align*}
 By applying Markov's inequality and Condition \eqref{assump:anticoncentration} as in the proof of \eqref{eq:dominate-block}, one finds 
 \begin{equation}\label{eq:dominate-other-blocks}
    \P\left( \min_{a_j \leq s \leq b_j} \left( \frac{ c_{m_j} e^{m_j s}|\eps_{m_j}|}{4} - \sum_{i \notin I_j} |\eps_{i}| c_i e^{is} \right) \leq 0 \right) =  e^{-\Omega_\alpha(j^2)}\,. 
 \end{equation}
 Combining bounds \eqref{eq:dominate-block} and \eqref{eq:dominate-other-blocks} shows \eqref{eq:anticoncentration-full}.
\end{proof}

We now want to show that the number of positive real roots is the same as the number of sign changes among the coefficients $(\eps_{m_j})_{j}$.  We first show that there are no real roots in the interval  $[a_j,b_j]$ with high probability.

\begin{lemma}\label{lem:no-root}
    There exist constants $C, c > 0$ (depending only on $\alpha$ and $C_0$) so that the following holds.  For each $j \geq 1$ with $m_j \leq n$, with probability at least $1 - C e^{-c j^2}$ we have that $\sgn(g_n(s)) = \sgn(\eps_{m_j})$ for all $s \in [a_j,b_j]$.  In particular, with probability at least $1 - C  e^{-c j^2}$, $g_n(s)$ has no roots in $[a_j, b_j]$.  Similarly, for all $s \geq a_{j_\ast}$ we have $\sgn(g_n(s)) = \sgn(\eps_{n})$ 
    with probability at least $1 - C e^{-c j_\ast^2}$. 
\end{lemma}
\begin{proof}
    Applying the triangle inequality, we see that 
    $$|g_n(s)| \geq  c_{m_j}e^{m_js} |\eps_{m_j}| - \sum_{i \neq m_j} |\eps_{i}|  c_i e^{is} \,. $$
    Thus, the conclusion follows from Lemma \ref{lemma:anticoncentration}.   The assertion for $s \geq a_{j_\ast}$ follows from Lemma \ref{lemma:anticoncentration} as well.
\end{proof}

We next show that if there is no sign change among $\{\eps_{m_j},\eps_{m_{j+1}}\}$ then there is no real root on $[b_j,a_{j+1}]$.

\begin{lemma}\label{lem:no-change}
There exist constants $C, c > 0$ (depending only on $\alpha$ and $C_0$) so that the following holds.  For each $j < j_\ast$ with $m_{j+1} \leq n$, with probability at least $1 - C e^{-c j^2}$, if $\sgn(\eps_{m_j})= \sgn(\eps_{m_{j+1}})$ then $g_n(t)$ has no roots in the interval $[b_j,a_{j+1}]$.  Similarly, if $m_{j_\ast} > n$ and $\sgn(\eps_{m_{j_\ast -1 }}) =  \sgn(\eps_{n})$ then $g_n$ has no roots in the interval $[b_{j_\ast - 1},a_{j_\ast}]$  with probability at least $1 - C e^{-c j_\ast^2}$.
\end{lemma}
\begin{proof}
    Assume without loss of generality that $\sgn(\eps_{m_j})= \sgn(\eps_{m_{j+1}}) = +1$.  Then for any $s \in [b_j,a_{j+1}]$, we bound 
    \begin{align*}
        g_n(s) &\geq \left( \frac{\eps_{m_j}c_{m_j} e^{m_j s}}{2} - \sum_{i < m_j} |\eps_{i}| c_i e^{is}\right) + \left( \frac{\eps_{m_{j+1}}c_{m_{j+1}} e^{m_{j+1} s}}{2} - \sum_{i > m_{j+1}} |\eps_{i}| c_i e^{is}\right) \\
        &\qquad + \left( \frac{\eps_{m_j}c_{m_j} e^{m_j s}}{2} + \frac{\eps_{m_{j+1}}c_{m_{j+1}} e^{m_{j+1} s}}{2} - \sum_{i = m_{j} + 1}^{m_{j+1}-1} |\eps_{i}| c_i e^{is} \right)
    \end{align*}
    The first two terms can be bounded using \eqref{eq:anticoncentration-left} and \eqref{eq:anticoncentration-right} from Lemma \ref{lemma:anticoncentration}.  The third term on the right-hand side can be bounded by following a nearly identical argument which led to \eqref{eq:dominate-block}.  The assertion for $m_{j_\ast} > n$ follows in a similar way from Lemma \ref{lemma:anticoncentration}. 
\end{proof}

We now handle the other case and show that if there is a sign change among $\{\eps_{m_j},\eps_{m_{j+1}}\}$ then there is precisely one real root on $[b_j,a_{j+1}]$.  For this, we show that after rescaling our function $g_n$, the derivative does not have a sign change with high probability.

\begin{lemma}\label{lem:yes-change}
There exist constants $C, c > 0$ (depending only on $\alpha$ and $C_0$) so that the following holds.  For each $j \geq 1$ with $m_{j+1} \leq n$, with probability at least $1 - C e^{-c j^2}$, if $\sgn(\eps_{m_j}) = -\sgn(\eps_{m_{j+1}})$, then $g_n$ has a single root in the interval $[b_j,a_{j+1}]$.  Similarly, if $m_{j_\ast} > n$ and $\sgn(\eps_{m_{j_\ast -1 }}) = - \sgn(\eps_{n})$ then $g_n$ has a single root in the interval $[b_{j_\ast - 1},a_{j_\ast}]$  with probability at least $1 - C e^{-c j_\ast^2}$.
\end{lemma}
\begin{proof}
    We prove the first assertion as the second is similar. 
    By replacing $g_n$ with $-g_n$ if necessary, assume without loss of generality that $\eps_{m_{j+1}} > 0$ and $\eps_{m_j} < 0$.  Write $g_n(t) = e^{t(m_{j+1} + m_j)/2} h_n(t)$ with 
    \begin{align*}
    	h_n(t) := \eps_{m_j} &c_{m_{j}}e^{-t(m_{j+1} - m_j)/2} + \eps_{m_{j+1}}c_{m_{j+1}} e^{t(m_{j+1}-m_j)/2} \\
	&+ \sum_{ i \notin \{m_j,m_{j+1}\} } \eps_{i}c_{i} e^{t(i - m_{j+1}/2 - m_j/2)}\,.   
    \end{align*}
    It is sufficient to show that $h_n(t)$ has precisely one root in $[b_j, a_{j+1}]$.  By Lemma \ref{lem:no-root}, $h_n$ exhibits a sign change on the interval $[b_j,a_{j+1}]$ with probability at least $1 - C e^{-c j^2}$.  Thus, it will be sufficient to show that with probability at least $1 - C  e^{-c j^2}$, $h_n'$ is positive on the interval.  We compute \begin{align*}
        \frac{2 e^{t(m_{j+1} + m_j)/2} }{m_{j+1} - m_j}h_n'(t) &= (-\eps_{m_j}) c_{m_j} e^{m_j t} + \eps_{m_{j+1}}  c_{m_{j+1}} e^{m_{j+1} t} \\
        &\qquad + \sum_{ i \notin \{m_j,m_{j+1}\} } \eps_{i} c_{i} e^{it}\cdot \frac{2i - m_{j+1} - m_j}{m_{j+1} - m_j}\,.
    \end{align*}

    Recalling that $(-\eps_{m_j}) > 0$ and $\eps_{m_{j+1}} > 0$, it suffices to show that 
    \begin{equation*}
         |\eps_{m_j}| c_{m_j} e^{m_j t} - \sum_{ i < m_j } |\eps_{i}| c_i e^{it} \left| \frac{2i - m_{j+1} - m_j}{m_{j+1} - m_j}\right| > 0
    \end{equation*}
    and 
    \begin{equation*}
        |\eps_{m_{j+1}}|c_{m_{j+1}} e^{m_{j+1} t} - \sum_{ i > m_j, i \neq m_{j+1} } |\eps_{i}| c_i e^{it} \left| \frac{2i - m_{j+1} - m_j}{m_{j+1} - m_j}\right| > 0 
    \end{equation*}
    with probability at least $1 - C e^{-c j^2}$.  
    Since 
    \[ \left| \frac{2i - m_{j+1} - m_j}{m_{j+1} - m_j} \right| \leq 2i + 2m_j \]
    (this follows from the triangle inequality since $m_{j+1} - m_j \geq 1$), the conclusion follows by applying the proof of Lemma \ref{lemma:anticoncentration} (with only slight changes to account for this extra factor of $2i + 2m_j$). 
\end{proof}

Since our previous lemmas are all asymptotic as $j \to \infty$, we now need to show that the behavior of the first few terms does not contribute too many roots close to the origin.  A simple application of Rouch\'e's theorem will allow us to bound the number of roots in a ball centered at $0$:

\begin{lemma}\label{lem:rouche}
There exist constants $C, c > 0$ (depending only on $\alpha$ and $C_0$) so that the following holds.  For each $j \geq 1$ with $m_j < n$, with probability at least $1 - C  e^{-c j^2}$, the polynomial $f_n(z)$ has at most $m_j$ roots in the disk $\{z \in \mathbb{C} : |z| \leq e^{b_j} \}$.    
\end{lemma}
\begin{proof}
    We will use Rouch\'e's theorem.  In particular, if we show that for all $z \in \C$ with $|z| = e^{b_j}$ we have  
    \begin{equation}
    |\eps_{m_j} c_{m_j} z^{m_j}| > \left|\sum_{i \neq m_j} \eps_{i} c_{i} z^{i}\right|,
    \end{equation}
    then we will have that $f_n(z)$ has at most $m_j$ roots in the disk $\{z \in \mathbb{C} : |z| \leq e^{b_j} \}$.  We lower bound \begin{align*}
       \min_{|z| = e^{b_j}} \left(|\eps_{m_j} c_{m_j} z^{m_j}| - \left|\sum_{i \neq m_j} \eps_{i} c_{i} z^{i}\right|\right) &\geq  |\eps_{m_j}| c_{m_j} e^{m_jb_j} - \sum_{i \neq m_j} |\eps_i| c_i e^{ib_j} 
    \end{align*}
    and hence the conclusion follows from Lemma \ref{lemma:anticoncentration}.  
\end{proof}

With the lemmas above in hand, we are now ready to complete the proof of Theorem \ref{thm:main}.  

\begin{proof}[Proof of Theorem \ref{thm:main}]
Choose $k = \log n$ and apply Lemma \ref{lem:rouche} to see that, with overwhelming probability, the number of real roots of $f_n$ differs from the number of real roots in $\{z \in \mathbb{R} : |z| > e^{b_k}\}$ by at most $2(\log n)^{1/\alpha}$.  Combining Lemmas \ref{lem:no-root}, \ref{lem:no-change} and \ref{lem:yes-change} shows that, with overwhelming probability, the number of real roots in $\{t \in \mathbb{R} : t > e^{b_k}\}$ is equal to the number of sign changes of the sequence $(\eps_{m_j})_{j = k}^{j_\ast - 1} \cup (\eps_n)$; similarly, the same argument shows that the number of \emph{negative} real roots in $\{t \in \mathbb{R} : t < - e^{b_k}\}$ is also equal to the number of sign changes of $(\eps_{m_j})_{j = k}^{j_\ast - 1} \cup (\eps_n)$.
This shows that $|R_n - 2S_n| \leq 4(\log n)^{1/\alpha}$ with overwhelming probability.  
    The other conclusions of Theorem \ref{thm:main} follow from this bound.  For example, the classical law of large numbers for $S_n$ implies the same conclusion for $R_n$.  
\end{proof}

\section*{Acknowledgements} 
The authors thank the anonymous referee for useful feedback and corrections.

\bibliographystyle{abbrv}
\bibliography{main.bib}

\end{document}